\documentclass[12pt, twoside, leqno]{article}

\usepackage{amsmath,amsthm}
\usepackage{amssymb}

\usepackage{enumitem}

\usepackage{graphicx}

\usepackage[T1]{fontenc}

\usepackage{hyperref}
\hypersetup{
  colorlinks=true,       
  linkcolor=magenta,          
    citecolor=green,        
    filecolor=magenta,      
    urlcolor=cyan           
}

\def\norm#1{\left\Vert#1\right\Vert}

\newtheorem{theorem}{Theorem}

\newtheorem{lemma}[theorem]{Lemma}

\theoremstyle{definition}

\numberwithin{equation}{section}

\begin{document}


\baselineskip=17pt


\title{A note on the groups of finite type and the Hartman--Mycielski construction}

\author{Vladimir G. Pestov\\
Instituto de Matem\'atica e Estat\'\i stica,\\
 Universidade Federal da Bahia,\\
 Ondina, Salvador, BA, 40.170-115, Brasil \\
{\em and} \\
Department of Mathematics and Statistics,\\
       University of Ottawa,\\
       Ottawa, ON, K1N 6N5, Canada\\
E-mail: vpest283@uottawa.ca}

\date{Version of June 14, 2020}

\maketitle


\renewcommand{\thefootnote}{}



\renewcommand{\thefootnote}{\arabic{footnote}}
\setcounter{footnote}{0}


\begin{abstract}
Ando, Matsuzawa, Thom, and T\"ornquist have resolved a problem by Sorin Popa by constructing an example of a Polish group of unitary operators with the strong operator topology, whose left and right uniform structures coincide, but which does not embed into the unitary group of a finite von Neumann algebra. The question remained whether such a group can be connected. Here we observe that a connected (in fact, homeomorphic to the Hilbert space) example is obtained from the example of the above authors via the Hartman--Mycielski construction. 
\end{abstract}

\section*{}

A Polish topological group is of {\em finite type} if it topologically embeds into the unitary group, $U(A)$, of a finite von Neumann algebra, equipped with the strong topology. Any such group $G$ is a {\em SIN group}, that is, a base of neighbourhoods at identity is formed by open sets $V$ invariant under conjugations: $g^{-1}Vg=V$ for all $g\in G$. Also, $G$ is {\em unitarily representable}, that is, it admits a unitary representation $\pi\colon G\to U({\mathcal H}_{\pi})$ that is an embedding of topological groups with regard to the strong operator topology. For a while it remained unclear if those two properties were enough to characterize the groups of finite type. The question was asked in print by Popa \cite{popa} and resolved in the negative in the article \cite{AMTT}, to which we refer for more detailed definitions and references. 

The central result of the article states that the semidirect product $\Gamma\ltimes_{\pi}\ell^2$ of a discrete group $\Gamma$ and the additive topological group of the Hilbert space $\ell^2$ with regard to a representation $\pi\colon\Gamma\to GL(\ell^2)$ is unitarily representable and SIN if and only if $\pi$ is bounded, and is of finite type if and only if $\pi$ is unitarizable. Since there are many known examples of bounded non-unitarizable representations, this construction provides a counter-example to Popa's question. However, the Polish group $\Gamma\ltimes_{\pi}\ell^2$ is disconnected, so the authors have asked whether a connected Polish group with the same combination of properties exists. 

Here we note that such an example is derived from
the example of the above authors through a classical construction going back to Hartman and Mycielski \cite{HM}. 
In this paper, it was shown that if $G$ is any topological group, then the group of all maps from the interval $[0,1]$ to $G$, constant on half-open intervals and taking finitely many values, equipped with the topology of convergence in Lebesgue measure, is a path-connected and locally path connected topological group containing $G$ as a topological subgroup formed by constant maps. It was further remarked in \cite{BM} that this group is contractible and locally contractible. 

One can form a larger group, consisting of all (equivalence classes of) strongly (or: Bourbaki) measurable maps $f\colon X\to G$ from a standard Lebesgue space $(X,\mu)$ to $G$, meaning that for every $\varepsilon>0$ there is a compact subset $K\subseteq X$ of measure $>1-\varepsilon$ on which $f$ is continuous. We denote this group $L^0(X,\mu,G)$. The topology is still the topology of convergence in measure, whose basis is formed by the sets
\[[V,\varepsilon] = \{f\in L^0(X,\mu,G)\colon \mu\{x\in X\colon f(x)\in V\}>1-\varepsilon\},\]
where $V$ runs over a neighbourhood base at the identity of $G$ and $\varepsilon >0$.
The group $L^0(X,\mu,G)$ contains the original group of Hartman--Mycielski as an everywhere dense subgroup. Even in this generality, the construction yields interesting results, see e.g. \cite{PS}.

If a neighbourhood of identity $V$ is conjugation invariant, then so is $[V,\varepsilon]$, so we have:

\begin{lemma}
If a topological group $G$ is SIN, then the group $L^0(X,\mu,G)$ is SIN as well.
\label{l:sin}
\end{lemma}

If $G$ is metrizable, so is the group $L^0(X,\mu;G)$. For example, if $d$ is a right-invariant metric generating the topology of $G$, then the following metric generates the topology of convergence in measure and is right-invariant (we follow Gromov's notation \cite{Gr}, p. 115):
\[\mbox{me}_1(f,g) =\inf\left\{\varepsilon>0\colon \mu\{x\in X\colon d(f(x),g(x))>\varepsilon\}<\varepsilon\right\}.\]
For $G$ separable metric, strongly measurable maps from $(X,\mu)$ to $G$ are of course just the $\mu$-measurable maps.
Here is one of the main results about the Hartman--Mycielski construction.

\begin{theorem}[Bessaga and Pe{\l}czy{\'n}ski \cite{BP}]
Whenever the group $G$ is separable metrizable, the topological group $L^0(X,\mu;G)$ is homeomorphic to the separable Hilbert space $\ell^2$ provided $G$ is nontrivial. 
\label{th:BP}
\end{theorem}


In fact, the above result was stated and proved for any separable metric space instead of $G$, but the motivation was to answer, in the affirmative, a question of Michael \cite{michael}: does every separable metrizable topological group embed into a topological group homeomorphic to $\ell^2$? 

The following observation is surely folklore, but we do not have a reference.  The argument is the same as in \cite{P10}, Lemma 6.5, for $G$ locally compact, and of course it could be generalized further if need be.

\begin{lemma} 
Let $G$ be a unitarily representable Polish group. Then $L^0(X,\mu;G)$ is unitarily representable as well.
\label{l:representable}
\end{lemma}

\begin{proof}
Fix a topological group embedding $\rho\colon G\to U({\mathcal H})$. 
Define a unitary representation $\pi$ of $L^0(X,\mu;G)$ in $L^2(X,\mu;{\mathcal H})$ as the direct integral of copies of the representation $\rho$: for every $f\in L^0(X,\mu;G)$ and $\phi\in L^2(X,\mu;{\mathcal H})$,
\[\pi_f(\psi)(x) =\rho_{f(x)}(\psi(x)),~~\mbox{ for $\mu$-a.e. }x.\]
For each simple function $\psi\in L^2(X,\mu;{\mathcal H})$, one can see that the corresponding orbit map 
\[L^0(X,\mu;G)\ni f\mapsto \pi_f(\psi)\in L^2(X,\mu;{\mathcal H})\]
is continuous at identity, and consequently $\pi$ is strongly continuous. 

Fix a right-invariant compatible pseudometric $d$ on $G$, and let $\varepsilon>0$ be arbitrary. There are $\xi_1,\ldots,\xi_n\in {\mathcal H}$ and a $\delta>0$ such that
\[\forall h\in G,~~\mbox{ if }\norm{\rho_g(\xi_i)-\xi_i}<\delta\mbox{ for all }i,\mbox{ then }d(g,e)<\varepsilon.\]
Let $\bar\xi_i$ be a constant function on $X$ taking value $\xi_i$. For any function $f\in L^0(X,\mu;G)$, if
\[\forall i=1,2,\ldots,n~~\norm{\pi_f(\bar\xi_i)-\bar\xi_i}<
\delta\sqrt{\frac{\varepsilon}{n}},\]
then one must have 
\[\forall i,~~\norm{\rho_{f(x)}(\xi_i)-\xi_i}<\delta\]
on a set of $x$ having measure at least $1-\varepsilon/n$. This implies $\mbox{me}_1(f,e)<\varepsilon$,
concluding the argument.
\end{proof}

Now it is enough to apply the construction to the Polish group $\Gamma\ltimes_{\pi}{\mathcal H}$ from \cite{AMTT}. The topological group $L^0(X,\mu;\Gamma\ltimes_{\pi}{\mathcal H})$ is homeomorphic to the Hilbert space $\ell^2$ (theorem \ref{th:BP}), in particular Polish and contractible, is SIN (lemma \ref{l:sin}), unitarily representable (lemma \ref{l:representable}), yet not of finite type, because it contains $\Gamma\ltimes_{\pi}{\mathcal H}$ as a closed topological subgroup made of constant functions.


\normalsize
\baselineskip=17pt


\end{document}